\documentclass[UTF-8,reqno,12pt]{amsart}
\usepackage{enumerate,amsmath,amssymb,amsthm}
\usepackage{amssymb,url,color}
\usepackage{hyperref}
\usepackage{amsmath,amsthm}
\usepackage{mathrsfs}
\usepackage[notref,notcite]{showkeys}%xianshigongshidaima
\usepackage{bm}

\allowdisplaybreaks[4]

\newtheorem{thm}{Theorem}[section]
\theoremstyle{Condition}

\newtheorem{cor}{Corollary}[section]
\newtheorem{prop}{Proposition}[section]
\newtheorem{lem}{Lemma}[section]
\newtheorem{rem}{Remark}[section]

\theoremstyle{Problem}

\theoremstyle{Assumption}

\theoremstyle{Definition}

\numberwithin{equation}{section}

\def\beq{\begin{equation}}
\def\deq{\end{equation}}
\allowdisplaybreaks %%%

\bfseries\rmfamily %%???

\def\cC{{\mathcal C}}
\def\cD{{\mathcal D}}

\def\cF{{\mathcal F}}

\def\cL{{\mathcal L}}

\def\cP{{\mathcal P}}

\def\mE{{\mathbb E}}

\def\mI{{\mathbb I}}

\def\mN{{\mathbb N}}

\def\mP{{\mathbb P}}

\def\mR{{\mathbb R}}

\def\mW{{\mathbb W}}

\def\geq{\geqslant}
\def\leq{\leqslant}

\def\eps{\varepsilon}
\def\t{\tau}

\def\a{\alpha}
\def\om{\omega}
\def\Om{\Omega}
\def\b{\beta}

\def\d{\delta}

\def\l{\lambda}
\def\s{\sigma}

\def\[{{\Big[}}
\def\]{{\Big]}}
\def\<{{\langle}}
\def\>{{\rangle}}
\def\({{\Big(}}
\def\){{\Big)}}

\def\sgn{\mbox{\rm sgn}}

\def\dif{{\rm d}}

\def\min{{\mathord{{\rm min}}}}

\def\={&\!\!=\!\!&}
\def\bt{\begin{theorem}}
\def\et{\end{theorem}}
\def\bl{\begin{lemma}}
\def\el{\end{lemma}}
\def\br{\begin{rem}}
\def\er{\end{rem}}

\begin{document}

\title[Jump-type McKean-Vlasov SDEs]
{Well-posedness and propagation of chaos for jump-type McKean-Vlasov SDEs with irregular coefficients}

\author{Zhen Wang$^{1\S}$, Jie Ren$^{2\dag}$ and Yu Miao$^{3\dag}$}

\thanks{$^{\S}$Corresponding author at: College of Mathematics and Information Science, Henan Normal University, Xinxiang, Henan $453007$, P. R. China. E-mail addresses$:$ wangzhen881025@163.com }
\thanks{$^{\dag}$These authors contributed equally to this work.}

\dedicatory{$^1$College of Mathematics and Information Science,
        Henan Normal University,\\
        Xinxiang, Henan $453007$, P. R. China.
        \\
        $^2$ College of Mathematics and Information Statistics, Henan University of Economics and Law, Zhengzhou, Henan, 450000, P.R.China.\\
$^3$ College of Mathematics and Information Science, Henan Normal University, \\
Xinxiang, Henan, 453007, P. R. China
}

\thanks{This work is partially supported by an NNSFC grant of China (No.11971154) and (No.11901154)}.

\begin{abstract}
\ \ In this paper, we study the existence and pathwise uniqueness of strong solutions for jump-type McKean-Vlasov SDEs
 with irregular coefficients but uniform linear growth assumption. Moreover, the propagation of chaos and the convergence rate for Euler's scheme of jump-type McKean-Vlasov SDEs
  are also obtained by taking advantage of Yamada-Watanabe's approximation approach and stopping time.
 \\
 Keywords.  \ \  Jump-type McKean-Vlasov SDEs, Propagation of chaos, Euler approximation, Irregular coefficients.\\
{\it AMS Mathematics Subject Classification $(2010)$$:$} 60H10, 60H15, 60H20.\\
\end{abstract}

\maketitle

\section{Introduction}
McKean-Vlasov stochastic differential equations (SDEs) are referred to as distribution dependent SDEs or mean field SDEs, which are expressed as limits
of interacting diffusions and which coefficients depend on the solution of the equation and on the law of this solution. However, the mean field interactions are derived as a dependency of the coefficients on the empirical measure of system, if $N$ goes to infinity, the empirical measure converges to the law of any particle of the limit system. After it was first proposed by McKean \cite{M}, using Kac's formalism of molecular chaos \cite{K}. There are endless studies on McKean-Vlasov SDEs (cf.\cite{CD,S}), including well-posedness (\cite{d,RZh,WRM}), ergodicity (\cite{LMW,W1}), Feynman-Kac formulae (\cite{BLPR,CM,RRW}) and the principle of large deviation (\cite{dSW,RWW}) ect. At the same time, McKean-Vlasov SDEs have been applied extensively in stochastic control, mathematical finance (\cite{Sh,T}), stochastic volatility \cite{CI} and so on.

The strong well-posedness and the propagation of chaos are often studied for McKean-Vlasov SDEs without jump. For instance, G\"arner \cite{G} verified these two properties assuming that the coefficients are Lipschitz continuous. Under the local Lipschitzian in the state variable, but uniform linear growth assumption, Li and so on \cite{LMW2} developed the existence and uniqueness of sulutiond of McKean-Vlasov SDEs.
About non-Lipschitz case, Ding and Qiao \cite{DQ} treated the strong well-posedness by the weak existence and pathwise
uniqueness. While, it is noticed that in \cite{BH}, the authors demonstrated the propagation of chaos and the convergence rate of Euler's scheme with irregular coefficients taking advantage of Yamada-Watanabe's approximation approach and Zvonkin's transformation. For additive noise case, the existence and uniqueness of strong solutions with irregular drift were established in \cite{BMP}. With singular $L_p$-interactions as well as for the moderate interaction particle systems on the level of particle trajectories, researchers \cite{HRZ} showed the strong convergence of the propagation of chaos for the particle approximation of McKean-Vlasov SDEs. Moreover, using a discretized version of Krylov's estimate, Zhang \cite{Zh2} studied the weak and strong convergences for Euler's approximation of McKean-Vlasov SDEs with non-degenerate conditions. Huang and Wang \cite{HW} proved the existence and uniqueness of McKean-Vlasov SDEs with non-degenerate noise, under integrability conditions on distribution dependent coefficients.

On the other hand, the study of McKean-Vlasov SDEs with jumps, there are still relatively few results to discuss. It is a classical fact that when coefficients are Lipschitz continuous in \cite{Gr}, the authors obtained some results about the strong well-posedness and the propagation of chaos for McKean-Vlasov SDEs with discrete jumps. In \cite{ADF}, the authors considered McKean-Vlasov SDEs with simultaneous jumps and proved the propagation of chaos in this model under globally Lipschitz assumptions.  When the coefficients are locally Lipschitz continuous, Xavier \cite{X} emphasized the strong well-posedness and the propagation of chaos property for McKean-Vlasov SDEs with simultaneous jumps.

Based on Heston's stochastic volatility model with jump-to-default \cite{PW} and Hybrid Heston-CIR Model with jumps \cite{BW} in mathematical finance, and since that the distribution of stochastic processes can be regarded as macroproperty, we carry out this research.
In this paper, the novely of our results is work on the McKean-Vlasov SDEs with discrete jumps under irregular coefficients, but uniform linear growth assumption. The questions about the strong well-posedness, the propagation of chaos and Euler's approximations have also been studied in this $1$-dimensional framework taking advantage of Yamada-Watanabe's approximation approach and stopping time. Namely, we promote Theorem 1.2 in \cite{BH} to results with jumps and $p$-th moment $(p>0)$, while the convergence rate of jump-type Euler's approximations just get similar results to Theorem 1.3 in \cite{BH}.

Let $\left(\Om,\cF,\cP;(\cF_t)_{t\geq 0}\right)$ be a complete filtration probability space, endowed with a standard $1$-dimensional Brownain motion $(W_t)_{t\geq 0}$ on the probability space. And $\{p_0(t)\}$ and $\{p_1(t)\}$ are two $\cF_t$-Poisson point processes on $U_0$ and $U_1$ with characteristic measures $\nu_0(\dif u)$ and $\nu_1(\dif u)$, such that $\{W_t\}$, $\{p_0(t)\}$ and $\{p_1(t)\}$ are independent of each other. Let $N_0(\dif s,\dif u)$ and $N_1(\dif s,\dif u)$ be Poisson random measures associated with $\{p_0(t)\}$ and $\{p_1(t)\}$, respectively. Suppose that $\mu$ is a probability measure on $\mR$ and $\mu_t$ stands for the law of $X_t$, the initial value $X_0=\xi$ is a $\nu$-distributed, and $\cF_0$-measurable $\mR$-valued random variable, $b=b_1+b_2$ and the coefficients
$$
b_i:\mR\times\cP(\mR)\rightarrow\mR,i=1,2;\ \  \s:\mR\rightarrow\mR\otimes\mR
$$
are Borel measurable functions, and $f_0:\mR\times\cP(\mR)\times U_0\rightarrow\mR$ and $f_1:\mR\times\cP(\mR)\times U_1\rightarrow\mR$ also are Borel measurable functions. Consider the following McKean-Vlasov stochastic differential equations:
\beq\label{a-1}
\aligned
 X_{t}=&\xi+\int_0^tb\left(X_{s},\mu_s\right)\dif s+\int_0^t\s\left(X_{s}\right)\dif W_s
 +\int_0^{t+}\int_{U_0}f_0(X_{s-},\mu_s,u)\widetilde{N}_0(\dif s,\dif u)\\&+\int_0^{t+}\int_{U_1}f_1(X_{s-},\mu_s,u)N_1(\dif s,\dif u),\\
\endaligned
 \deq
where
$$
\widetilde{N}_0(\dif t,\dif u)=N_0(\dif t,\dif u)-\nu_0(\dif z)\dif t
$$
is the compensated Poisson random measure of $N_1(\dif t,\dif u)$. At the same time, $X_t\neq X_{t-}$ for at most countably many $t>0$, $b(X_{t-}, \mu_t)$ and $\s(X_{t-}, \mu_t)$ can be replaced by $b(X_{t}, \mu_t)$ and $\s(X_{t}, \mu_t)$ for fixed the law $\mu_t$, respectively.
Let $U_2$ be a Borel subset of $U_1$ satisfying $\nu_1(U_1\backslash U_2)<\infty$. Consider that the stochastic equation with jump:
\beq\label{1-a}
\aligned
X_{t}=&\xi+\int_0^tb\left(X_{s},\mu_s\right)\dif s+\int_0^t\s\left(X_{s}\right)\dif W_s
 +\int_0^{t+}\int_{U_0}f_0(X_{s-},\mu_s,u)\widetilde{N}_0(\dif s,\dif u)\\&+\int_0^{t+}\int_{U_2}f_1(X_{s-},\mu_s,u)N_1(\dif s,\dif u),\\
 \endaligned
\deq
\begin{prop}\cite{FL}
Then (\ref{a-1}) has a strong solution if (\ref{1-a}) has a strong solution. And if the pathwise uniqueness of solutions holds for (\ref{1-a}), it also holds for (\ref{a-1}).
\end{prop}

In the present paper,
we make the following assumptions: Supposing that for any $x,y\in\mR$, and $\mu,\nu\in\cP(\mR)$, we have\\
$\mathbf{(H_1)}$  For some constant $K_1>0$ and $\b\in(0,1]$,
$$
\left|b_1(x,\mu)-b_1(x,\nu)\right|\leq K_1\mW_2(\mu,\nu),
$$
$$
\left|b_1(x,\mu)-b_1(y,\mu)\right|\leq K_1|x-y|^{\b},
$$
$$
\left|b_2(x,\mu)-b_2(y,\nu)\right|\leq K_1\left(|x-y|+\mW_2(\mu,\nu)\right),
$$
where $x\rightarrow b_1(x,\mu)$ is continuous and non-increasing.   \\
$\mathbf{(H_2)}$ For any $\alpha\in[1/2,1]$,
$$
|\s(x)-\s(y)|\leq K_2\left|x-y\right|^{\alpha},
$$
where $K_2$ is a positive constant.\\
$\mathbf{(H_3)}$ There exists the positive constant $K_3$ such that
$$
\aligned
&\int_{U_0}\max\left\{|x-y|\left|f_0(x,\mu,u)-f_0(y,\nu,u)\right|,\left|f_0(x,\mu,u)-f_0(y,\nu,u)\right|^{2}\right\}\nu_0(\dif u)\\
\leq &K_3 \left(|x-y|^2+\mW_{2}(\mu,\nu)^2\right),\\
&\int_{U_2}\max\left\{|x-y|\left|f_1(x,\mu,u)-f_1(y,\nu,u)\right|,\left|f_1(x,\mu,u)-f_1(y,\nu,u)\right|^{2}\right\}\nu_(\dif u)\\
\leq &K_3 \left(|x-y|^2+\mW_{2}(\mu,\nu)^2\right).\\
\endaligned
$$
$\mathbf{(H_4)}$ (Linear growth conditions) There are the positive constants $M_1$ and $M_2$,
such that for any $t\in[0,T]$, $x\in\mR$ and $\mu\in\cP(\mR)$,
$$
|b(x,\mu)|^2\vee\|\s(x,\mu)\|^2\leq M_1(1+|x|^2+\mW_2(\mu,\delta_0)^2),
$$
$$
\aligned
&\int_{U_0}\left|f_0(x,\mu,u)\right|\wedge\left|f_0(x,\mu,u)\right|^2\nu_0(\dif u)
\leq M_2(1+|x|+\mW_2(\mu,\delta_0)),\\
&\int_{U_2}\left|f_1(x,\mu,u)\right|\nu_1(\dif u)
\leq M_3(1+|x|+\mW_2(\mu,\delta_0)),\\
\endaligned
$$
where $\delta_0$ denotes the Dirac measure at $0$.\\

On the above assumptions, our first result is the strong well-posedness and properties of $p$-th moment $(p>0)$.
\begin{thm}\label{thm-1}
Let suppose that Assumptions $\mathbf{(H_1)}$-$\mathbf{(H_4)}$ hold. For any $\cF_0$-measurable random variable $\xi$ satisfying $\mE|\xi|^{\beta}<\infty$ for any $\beta>0$, then (\ref{a-1}) has a unique strong solution $\{X_t\}_{t\geq 0}\in L^2(\Omega,\mR)$ for $t\in[0,T]$ with the initial value $\xi$. For any $p>0$, we have
\beq\label{a-2}
\mE\left(\sup_{t\in[0,T]}|X_t|^p\right)< \infty.
\deq
\end{thm}

The second main result is the propagation of chaos property of jump-type McKean-Vlasov SDEs under the same irregular assumptions.

Suppose that the sequence of $\{\xi_i,i=1,2,\cdots,N\}$ is independent identically distributed with a common distribution $\nu$ in $\mR$ and $\{W^i,i=1,2,\cdots,N\}$ is a sequence of independent $1$-dimensional Brownian motions. And consider
the following non-interacting particle systems:
$$
\aligned
 \dif X_{t}^i=&b\left(X_{t}^i,\mu_t^i\right)\dif t+\s(X_{t}^i)\dif W_t^i+\int_{U_0}f_0(X_{t-}^i,\mu_t^i,u)\widetilde{N}_0(\dif t,\dif u)\\
 &+\int_{U_2}f_1(X_{t-}^i,\mu_t^i,u)N_1(\dif t,\dif u),\ \ X_0^i=\xi_i, \\
\endaligned
$$
where $\mu_{t}^i$ stands for the law of $X_{t}^i$, and $\{X_{\cdot}^i, i=1,2,\cdots,N\}$ is a set of independent identically distributed stochastic processes with the common distribution $\mu_{\cdot}$. By the weak uniqueness due to Theorem \ref{thm-1}, we known that $\mu_t=\mu_t^i,i=1,2,\cdots,N$. Defined $\widetilde\mu_t^N$ as the empirical distribution associated with $X_t^1,X_t^2,\cdots,X_t^N$, that is,
$$
\widetilde\mu_t^N=\frac{1}{N}\sum_{j=1}^N\d_{X_t^j}.
$$

Next, we use the  discretized $N$-interacting particle to approximate the above non-interacting particle. Namely, consider $N$-interacting particle $X^{N,i}$, satisfying
\beq\label{a-9}
\aligned
\dif X_{t}^{N,i}=&b\left(X_{t}^{N,i},\hat{\mu}_t^N\right)\dif t+\s(X_{t}^{N,i})\dif W_t^i+\int_{U_0}f_0(X_{t-}^{N,i},\hat{\mu}_t^N,u)\widetilde{N}_0(\dif t,\dif u)\\
 &+\int_{U_2}f_1(X_{t-}^{N,i},\hat{\mu}_t^N,u)N_1(\dif t,\dif u),\ \ X_0^{N,i}=\xi_i,\\
\endaligned
\deq
where $\hat{\mu}_t^N$ means the empirical distribution corresponding to $X_t^{N,1},X_t^{N,2},\cdots,X_t^{N,N}$, i.e.,
$$
\hat{\mu}_t^N=\frac{1}{N}\sum_{j=1}^N\d_{X_t^{N,j}}.
$$

More precisely, the main statement is as following:
\begin{thm}\label{thm-2}Suppose that Assumptions $\mathbf{(H_1)}$--$\mathbf{(H_4)}$ hold. Then for any $T>0$ and for $p=1,2$ such that
\beq\label{1-4}
\lim_{N\rightarrow\infty}\sup_{i=1,2,\cdots,N}\mE\left[\sup_{0\leq t\leq T}\left|X_t^{N,i}-X_t^i\right|^p\right]=0.
\deq
%In particular,
%\beq\label{1-5}
%\sup_{i=1,2,\cdots,N}\mE\left\|X_t^{N,i}-X_t^i\right\|_{\infty,T}<C_T\left\{N^{-\frac{1}{8}}\mathbf{1}_{\{\a=\frac{1}{2}\}}
%+N^{-\frac{2\a-1}{4}}\mathbf{1}_{\{\a\in(\frac{1}{2},1]\}}\right\}.
%\deq
%and
%\beq\label{1-6}
%\sup_{i=1,2,\cdots,N}\mE\left\|X_t^{N,i}-X_t^i\right\|_{\infty,T}^2<C_TN^{-\frac{1}{4}}.
%\deq
\end{thm}
\begin{rem}In this theorem, we extend the result of Lemma 3.2 in \cite{BH} to jump-type McKean-Vlasov SDEs. But unfortunately we didn't get the similar case for any $p>0$.
\end{rem}
Another goal of this paper is to get the corresponding overall convergence rate. To discretize (\ref{a-9}) in time, for fixed $h\in(0,1)$, the corresponding Euler's scheme is
$$
\aligned
\dif X_{t}^{h,N,i}=&b\left(X_{t_h}^{h,N,i},\hat{\mu}_{t_h}^{h,N}\right)\dif t+\s\left(X_{t_h}^{h,N,i}\right)\dif W_t^i+\int_{U_0}f_0(X_{{t_h}}^{h,N,i},\hat{\mu}_{t_h}^{h,N},u)\widetilde{N}_0(\dif t,\dif u)\\
 &+\int_{U_2}f_1(X_{{t_h}}^{h,N,i},\hat{\mu}_{t_h}^{h,N},u)N_1(\dif t,\dif u),\ \ X_0^{h,N,i}=\xi_i.\\
\endaligned
 $$
where $t_h:=[t/h]h$ and $\hat{\mu}_{t_h}^{h,N,i}=\frac{1}{N}\sum_{j=1}^N\d_{X_{t_h}^{h,N,i}}$.

\begin{thm}\label{thm-3}Under the assumptions of Theorem \ref{thm-1}, for $h\in(0,1)$ sufficiently small, for any $\a\in[\frac{1}{2},1]$, there exists the constant $C$ dependent of $T$ such that
$$
\sup_{i=1,2,\cdots,N}\mE\left[\sup_{0\leq t\leq T}\left|X_t^{h,N,i}-X_t^i\right|\right]<C
\left\{
\begin{aligned}
&N^{-\frac{1}{4}}+\left(\ln\frac{1}{h}\right)^{-1/2},\ \ \a=\frac{1}{2};\\
&N^{-\frac{2\a-1}{2}}+h^{\frac{(2\a-1)^2}{2}}+h^{\frac{\b(2\a-1)}{2}},\ \ \a\in\left(
\frac{1}{2},1\right]. \\
\end{aligned}
\right.
$$
$$
\sup_{i=1,2,\cdots,N}\mE\left[\sup_{0\leq t\leq T}\left|X_t^{h,N,i}-X_t^i\right|^2\right]<C
\left\{
\begin{aligned}
&N^{-\frac{1}{2}}+\left(\ln\frac{1}{h}\right)^{-1},\ \ \a=\frac{1}{2};\\
&N^{-(\a-1)}+h^{\frac{2\a-1}{2}}+h^{\frac{\b}{2}},\ \ \a\in\left(
\frac{1}{2},1\right]. \\
\end{aligned}
\right.
$$
\end{thm}

This paper is organized as follows. In Section 2, we introduce the main lemmas for later use. Section 3 yields the strong wellposedness of (\ref{a-1}). In section 4 and 5, we construct the propagation of chaos and the corresponding overall convergence rate, respectively.

Throughout the paper, $C$ with or without
indices will denote different positive constants (depending on the indices), whose values may change from one place to another and not important.
\section{Preliminary}

Firstly, in order to obtain our main results, we also need the following significant
 lemma.
\begin{lem}\label{b-2} \cite{OS}Suppose that $X(t)\in\mR$ is an {\it It\^o-L\'evy process} of the following form:
$$
\dif X(t)=b(t,\om)\dif t+\s(t,\om)\dif W_t+\int_{\mR}\gamma(t,u,\om)\overline{N}(\dif t,\dif u),
$$
where
$$
\overline{N}(x)=
\left\{
\begin{aligned}
&N(\dif t,\dif u)-\nu(\dif u)\dif t\ \ \mathrm{if}\ \ |u|<R; \\
&N(\dif t,\dif u),\ \ \mathrm{if}\ \  |u|\geq R, \\
\end{aligned}
\right.
$$
for some $R\in[0,\infty).$ Let $f\in C^{1,2}([0,T]\times\mR;\mR)$ and define $Y(t)=f(t,X(t))$. Then $Y(t)$ is again an It\^o-L\'evy process and
$$
\aligned
Y(t)=&\frac{\partial f}{\partial t}(t,X(t))\dif t+\frac{\partial f}{\partial x}(t,X(t))[b(t,\om)\dif t+\s(t,\om)\dif W_t]+\frac{1}{2}\s^2(t,X(t))\frac{\partial^f}{\partial x^2}(t,X(t-)) \\
&+\int_{|u|<\mR}\left\{f(t,X(t-)+\gamma(t,u))-f(t,X(t-))-\frac{\partial f}{\partial x}(t,X(t-))\gamma(t,u)\right\}\nu(\dif u)\dif t\\
&+\int_{\mR}\left\{f(t,X(t-)+\gamma(t,u))-f(t,X(t-))\right\}\overline{N}(\dif t,\dif u).\\
\endaligned
$$
\end{lem}

Next, we introduce a well-studied metric on the space of
distributions known as the Wasserstein distance which allows us to consider $\cP(\mR)$ as a metric space.

For $\mu,\nu\in\cP(\mR^d)(d\geq 1)$ and $p>0$, the $\mW_p$-Wasserstein distance between $\mu$ and $\nu$ is defined by
$$
\mW_p(\mu,\nu)=\inf_{\pi\in\cC(\mu,\nu)}\left(\int_{\mR^d\times \mR^d}|x-y|^p\pi(\dif x,\dif y)\right)^{\frac{1}{1\vee p}},
$$
where $\cC(\mu,\nu)$ is the set of all couplings of $\mu$ and $\nu$ on $\cP(\mR^d\times \mR^d)$, such that $\pi\in\cC(\mu,\nu)$ if and only if $\pi(\cdot,\mR^d)=\mu(\cdot)$ and $\pi(\mR^d,\cdot)=\nu(\cdot)$.
%denotes the collection of all probability measures on  $\mR^d\times \mR^d$ with marginals $\mu$ and $\nu$ on the first and second factors.
%Here, the formula (\ref{2-a}) is a complete metric on $\cC(\mu,\nu)$, which induces the topology of weak convergence in \cite{D}.

The topology induced by Wasserstein metric coincides with the topology of weak convergence of measure and
the convergence of all moments of order up to $p$ in \cite{D}.

Nevertheless, in the process of study for the propagation of chaos, it is primarily concerned with Wasserstein distance between $\mu^{i}_s$ and $\widetilde{\mu}_s^{N}$. In Corollary 4.1 of \cite{FG}, the author implied that
\begin{lem}\label{b-1}\cite{FG} Let $\mu\in\cP(\mR^d)(d\geq 1)$. Assume that $\int_{\mR^d}|x|^q\mu(\dif x)<\infty$ for any $p,q>0$ and for some $q>p$. There exists a constant $C$ depending only on $p,d,p$ such that, for any $N\geq 1$,
$$
\mE\left[\mW_p(\mu^{i}_s,\widetilde{\mu}_s^{N})^p\right]\leq C\mE[|X_s^i|^q]^{p/q}  \left\{
\begin{aligned}
&N^{-1/2}\leq CN^{-d/2},\ \ \mathrm{if}\ \  p>\frac{d}{2},q\neq 2p; \\
&N^{-1/2}\log(1+N)\leq CN^{-d/2}\log(1+N),\ \ \mathrm{if}\ \  p=\frac{d}{2},q\neq 2p;  \\
&N^{-p}\leq CN^{-p},\ \ \mathrm{if}\ \ p\in\left(0,\frac{d}{2}\right),q\neq \frac{d}{1-p}.\\
\end{aligned}
\right.
$$
\end{lem}

\section{Proof of Theorem \ref{thm-1}}

Firstly, in the proof process of topic, with the help of the Yamada-Watanabe approximation approach, we list the relevant assumptions, as following: for $\l>1$ and $\varepsilon\in(0,1)$, there is a continuous function $\phi_{\l,\varepsilon}:=\mR_+\rightarrow\mR_+$ with the support $[\varepsilon/\l,\varepsilon]$ so that
$$
0\leq \phi_{\l,\varepsilon}(x)\leq \frac{2}{x\ln\l},\ \ x>0,\ \ \int_{\eps/\l}^{\eps}\phi_{\l,\eps}(r)\dif r=1.
$$
We define the following mapping
\beq\label{c-2}
\mR\ni x\rightarrow V_{\l,\eps}(x):=\int_0^{|x|}\int_0^y\phi_{\l,\eps}(z)\dif z\dif y,
\deq
which is twice differentiable and satisfies
\beq\label{c-3}
|x|-\eps\leq V_{\l,\eps}(x)\leq |x|,\ \  \sgn(x)V'_{\l,\eps}(x)\in[0,1],\ \  x\in\mR,
\deq
 and
\beq\label{c-4}
0\leq V''_{\l,\eps}(x)\leq \frac{2}{|x|\ln\l}\mathbf{1}_{[\varepsilon/\l,\varepsilon]}(|x|),\ \  x\in\mR.
\deq
In (\ref{c-3}) and (\ref{c-4}), $\sgn(\cdot)$ means the sign function, and $V'_{\l,\eps}$ and $V''_{\l,\eps}$ are treated as respectively the first and second-order derivative of $V_{\l,\eps}.$ In the course of the following research, we take $\l=e^{\frac{1}{\varepsilon}}$ in (\ref{c-2}), that is $V_{\varepsilon}:=V_{e^{\frac{1}{\varepsilon}},\varepsilon}$.

Secondly, we introduce the distribution-iterated scheme of (\ref{a-1}): for any $k\in \mN_+$,
\beq\label{c-6}
\aligned
\dif X_t^{(k)}=&b(X_t^{(k)},\mu_t^{(k-1)})\dif t+\s(X_t^{(k)})\dif W_t+\int_{U_0}f_0(X_{t-}^{(k)},\mu_t^{(k-1)},u)\widetilde{N}_0(\dif t,\dif u)\\&+\int_{U_2}f_1(X_{t-}^{(k)},\mu_t^{(k-1)},u)N_1(\dif t,\dif u)
\endaligned
\deq
where $\mu_t^{(k)}=\cL_{X_t^{(k)}}$.

Next, we need to prove that the $p$--order moment ($p>0$) is uniformly bounded in a finite time interval.

\begin{lem}\label{lem-1}Under Assumptions $\mathbf{(H_1)}$--$\mathbf{(H_4)}$, for any $k\in \mN_+$, $p> 0$ and $\mE|\xi|^p<\infty$, there exists the constant $C$ only dependent on $T$ and $p$. The above distribution-iterated SDE (\ref{c-6}) satisfy the following relationships:
\beq\label{c-1}
\mE\left(\sup_{t\in[0,T]}|X_t^{(k)}|^p\right)<\infty.
\deq
\end{lem}
\begin{proof}
Under Assumptions $\mathbf{(H_1)}-\mathbf{(H_3)}$, for any $x\in\mR$ and $\mu\in\cP(\mR)$, it is easy to see that,
\beq\label{c-8}
|b(x,\mu)|\leq 2K_1(|x|+\mW_2(\mu,\d_0)+c_1),
\deq
\beq\label{c-13}
 |\s(x)|^2\leq 2K_2^2
 (|x|^{2 }+c_2)
 \deq
\beq\label{c-9}
\aligned
&\int_{U_0}\max\left\{|x||f_0(x,\mu,u)|,|f_0(x,\mu,u)|^2\right\}\nu_0(\dif u) \leq  K_3\left(|x|^{2}+\mW_2(\mu,\d_0)^2+c_3\right),\\
&\int_{U_2}\max\left\{|x||f_1(x,\mu,u)|,|f_1(x,\mu,u)|^2\right\}\nu_1(\dif u) \leq  K_3\left(|x|^{2}+\mW_2(\mu,\d_0)^2+c_3\right),\\
\endaligned
\deq
where $c_1=|b(0,\d_0)|$, $c_2=2|\s(0)|^2$ and $c_3=\max\{\int_{U_0}f(0,\d_0,u)^2\nu_0(\dif u),\int_{U_2}f(0,\d_0,u)^2\nu_1(\dif u)\}$.

Defined the stopping time, for any $R>0$,
$
\t_R:=\inf\left\{t\geq 0: |X_t^{(k)}|\}\geq R,k=1,2,\cdots \right\},
$
For any $p\geq2$, using BDG's inequality, we can get
$$
\aligned
&\mE\left\|X^{(1)}\right\|^p_{\infty,t\wedge \t_R}\\
\leq&5^{p-1}\mE\left(|\xi|^p\right)+5^{p-1}t^{p-1}\mE\left[\int_0^{t\wedge \t_R}|b(X_s^{(1)},\mu_s^{(0)})|^p\dif s\right]
+\frac{1}{2}5^{p-1}\mE\left[\int_0^{t\wedge \t_R}|\s(X_s^{(1)})|^2\dif s\right]^{p/2}\\
&+\frac{1}{2}5^{p-1}\mE\left[\int_0^{(t\wedge \t_R) +}\int_{U_0}\left|f_0(X_{s-}^{(1)},\mu_s^{(0)},u)\right|^2\nu_0(\dif u)\dif t\right]^{p/2}
\\
&+5^{p-1}\mE\left[\int_0^{(t\wedge \t_R) +}\int_{U_2}\left|f_1(X_{s-}^{(1)},\mu_s^{(0)},u)\right|^2N_1(\dif t,\dif u)\right]^{p/2}\\
\endaligned
$$
where the fourth term can be verified by the compensated Poisson random measure and the last item is Jensen's inequality. Using
(\ref{c-8}) (\ref{c-13}) and (\ref{c-9}), for some constants $C_1,C_2$, it holds that
$$
\aligned
&\mE\left\|X^{(1)}\right\|^p_{\infty,t\wedge \t_R}\\
\leq& C_1(1+t^{p-1})\left\{1+\mE|\xi|^p+\mE\int_0^{t\wedge \t_R}\{|X^{(1)}_s|^p+\mW_2(\nu_s^{(0)},\delta_0)^p\}\dif s\right\}\\
\leq & C_2(1+t)^p\left(1+\mE|\xi|^p\right)+C_2(1+t^{p-1})\int_0^{t}\mE|X^{(1)}_{s\wedge \t_R}|^p\dif s
\endaligned
$$
where the second inequality is true, applying the properties of $\mW_p$-Wasserstein distance:
 $$
  \sup_{t\in[0,T]}\mW_2(\mu_s^{(0)},\d_0)^p\leq\mE|\xi|^p,
 $$
and using Gr\"onwall's inequality, we deduce
$$
\mE\left(\sup_{t\in[0,T]}|X^{(1)}_{t\wedge \t_R}|^p\right)\leq C_pe^{C_p T}(1+\mE|\xi|^p).
$$
Also because
$$
\mP(\t_R\leq T)=\frac{\mE(|X^{(1)}_{ \t_R}|^p\mathrm{I}_{\{\t_R\leq T\}})}{R^p}\leq\frac{\mE|X^{(1)}_{T\wedge \t_R}|^p }{R^p}\leq\frac{C_T(1+\mE|\xi|^p)}{R^p}.
$$
Thus, using Borel-Cantelli lemma, the series $\mP(\t_R\leq T)=0\ \ a.s.,$ that is, $\lim_{R\rightarrow\infty}\t_R=:\t_{\infty}>T$ a.s. With the arbitrariness of $T$, it holds that $\t_{\infty}=\infty$, a.s. So
(\ref{c-1}) holds with making use of Fatou's lemma, for any $p\geq 2$.

 Next, using the triple $(X^{(n+1)},X^{(n)},\mu^{(n)})$ in lieu of $(X^{(1)},X^{(0)},\mu^{(0)})$, we can imply that (\ref{c-1}) still holds true for $n=k+1$ once (\ref{c-1}) is valid for some $n=k$.

 Moreover, by virtue of the linear growth conditions $\mathbf{(H_4)}$, for $0<p<2$, using $\mE[\sup_{t\in[0,T]}|X_t|^p]\leq (1+\mE|\xi|^2)^{p/2}e^{pCT}$, the similar conclusion is true. Thus, for any $p>0$, Eq. (\ref{c-1}) can be established.
\end{proof}
\begin{proof}[Proof of Theorem \ref{thm-1}]The proof of this theorem consists of four steps:\\
Step $1$: Based on the iteration scheme (\ref{c-6}) and Lemma \ref{lem-1}, we show that
$$ \lim_{k\rightarrow\infty}\mE\left[\sup_{0\leq t\leq T}\left|X_t^{(k+1)}-X_t^{(k)}\right|^2 \right]=0
$$

For notation brevity, setting $Z_t^{(k+1)}:=X_t^{(k+1)}-X_t^{(k)}$, $\Delta_{f_i}^{(k+1)}=f_i(X_{s-}^{(k+1)},\mu_s^{(k)},u)-f_i(X_{s-}^{(k)},\mu_s^{(k-1)},u)$, $i=0,1$. By means of It\^o formula, for any $\l\geq 0$ and $k=1,2,\cdots$ we have
$$
e^{-\l t}V_{\l,\varepsilon}\left(\left|Z_{t}^{(k+1)}\right|^2\right)=\sum_{i=1}^7\mI_{i,\varepsilon}^{\l}(t),
$$
where
$$
\aligned
\mI_{1,\varepsilon}^{\l}(t):=&-\l \int_0^{t}e^{-\l s}V_{\l,\varepsilon}(|Z_s^{(k+1)}|^2)\dif s\\
\mI_{2,\varepsilon}^{\l}(t):=&\int_0^{t}e^{-\l s}DV_{\l,\varepsilon}(|Z_s^{(k+1)}|^2)\left(b(X_s^{(k+1)},\mu_s^{(k)})-b(X_s^{(k)},\mu_s^{(k-1)})\right)\dif s\\
\mI_{3,\varepsilon}^{\l}(t):=&\int_0^{t}e^{-\l s}DV_{\l,\varepsilon}(|Z_s^{(k+1)}|^2)\left(\s(X_s^{(k+1)})-\s(X_{s}^{(k)})\right)\dif W_s\\
\mI_{4,\varepsilon}^{\l}(t):=&\frac{1}{2}\int_0^{t}e^{-\l s}D^2V_{\l,\varepsilon}(|Z_s^{(k+1)}|^2)\left|\s(X_s^{(k+1)})-\s(X_{s}^{(k)})\right|^2\dif s\\
\mI_{5,\varepsilon}^{\l}(t):=&\int_0^{t+}e^{-\l s}\int_{U_0}\left[V_{\l,\varepsilon}\left(\left|Z_{s-}^{(k+1)}+\Delta_{f_1}^{(k)}\right|^2\right)
-V_{\l,\varepsilon}(|Z_{s-}^{(k+1)}|^2)
-DV_{\l,\varepsilon}(|Z_{s-}^{(k+1)}|^2)\Delta_{f_1}^{(k+1)}\right]
\nu_0(\dif u)\dif s\\
\mI_{6,\varepsilon}^{\l}(t):=&\int_0^{t+}e^{-\l s}\int_{U_0}\left[V_{\l,\varepsilon}\left(\left|Z_{s-}^{(k+1)}+\Delta_{f_1}^{(k+1)}\right|^2\right)
-V_{\l,\varepsilon}(|Z_{s-}^{(k+1)}|^2)\right]\widetilde{N}_0(\dif s,\dif u)\\
\mI_{7,\varepsilon}^{\l}(t):=&\int_0^{t+}e^{-\l s}\int_{U_2}\left[V_{\l,\varepsilon}\left(\left|Z_{s-}^{(k+1)}+\Delta_{f_2}^{(k+1)}\right|^2\right)
-V_{\l,\varepsilon}(|Z_{s-}^{(k+1)}|^2)\right]N_1(\dif s,\dif u),\\
\endaligned
$$
and where for any $x\in\mR$,
$$
\aligned
DV_{\varepsilon}(|x|^2):=&V'_{\l,\varepsilon}(|x|^2)2x\\
D^2V_{\varepsilon}(|x|^2):=&4V''_{\l,\varepsilon}(|x|^2)|x|^{2}+2V'_{\l,\varepsilon}(|x|^2).\\
\endaligned
$$
With the help of (\ref{c-3}), we obtain that
$$
\mI_{1,\varepsilon}^{\l}\leq \l\varepsilon t-\l\int_0^{t}e^{-\l s}|Z_s^{(k+1)}|^2\dif s.
$$
Applying (\ref{c-3}) and recalling that $x\rightarrow b_1(x,\cdot)$ is non-increasing , we infer
\beq\label{3-1}
(x-y)\left(b_1(x,\cdot)-b_1(y,\cdot)\right)\leq 0,\ \ x,y\in\mR.
\deq
%That is,
%$$
%(X_s^{(k+1)}-X_s^{(k)})\left(b_1(X_s^{(k+1)},\mu_s^{(k)})-b_1(X_s^{(k)},\mu_s^{(k)})\right)\leq 0.
%$$
Taking advantage of above inequality, $|V'_{\l,\varepsilon}(x)|\leq 1$, $b=b_1+b_2$ and the assumption $\mathbf{(H_1)}$
, we lead to
$$
\mI_{2,\varepsilon}^{\l}
\leq 4K_1\int_0^{t}e^{-\l s}\left\{|Z_s^{(k+1)}|^2+\mW_2(\mu_s^{(k)},\mu_s^{(k-1)})^2\right\}\dif s.\\
$$

There is some constants $c_{\l}=\{\mathbf{1}_{\{\l=0\}}+\frac{1}{\l}\mathbf{1}_{\{\l>0\}}\}$, one obviously has
$$
\aligned
\mI_{4,\varepsilon}^{\l}\leq&4K_2^2\varepsilon \int_0^{t}e^{-\l s}|Z_s^{(k+1)}|^{2\a }I_{[\frac{\varepsilon}{\l},\varepsilon]}(|Z_s^{(k+1)}|^{2})\dif s+K_2^2 \int_0^{t}e^{-\l s}|Z_s^{(k+1)}|^{2\a}\dif s\\
\leq&4c_{\l}K_2^2t  \varepsilon^{2}+K_2^2\int_0^{t}e^{-\l s}|Z_s^{(k+1)}|^2\dif s.\\
\endaligned
$$

For $\mI_{5,\varepsilon}^{\l}$, since that
 $(x+y)^{2}\leq 2(x^{2}+y^{2})$, for all $x,y\geq 0$, there exists $\xi_2$ between $\min\left\{|Z_{s-}^{(k+1)}|,|Z_{s-}^{(k+1)}+\Delta_{f_i}^{(k+1)}|\right\}$ and $\max\left\{|Z_{s-}^{(k+1)}|,|Z_{s-}^{
(k+1)}+\Delta_{f_i}^{(k+1)}|\right\}$.  By Lagrange's mean value theorem, thus
$$
\aligned
&V_{\l,\varepsilon}\left(\left|Z_{s-}^{(k+1)}+\Delta_{f_1}^{(k+1)}\right|^2\right)-V_{\l,\varepsilon}\left(\left|Z_{s-}^{(k+1)}\right|^2\right)\\
\leq &2|V'_{\l,\varepsilon}(|\xi_2|^2)|\cdot|\xi_2|\cdot |Z_{s-}^{(k+1)}+\Delta_{f_i}^{(k+1)}-Z_{s-}^{(k+1)}| \\
\leq &2(|Z_{s-}^{(k+1)}|+|\Delta_{f_i}^{(k+1)}|)|\Delta_{f_i}^{(k+1)}|\\
\leq &3(|Z_{s-}^{(k+1)}|^{2}+|\Delta_{f_i}^{(k+1)}|^2), \\
\endaligned
$$
where the second inequality from
$$
0\leq |\xi_2|\leq \max\{|Z_{s-}^{(k+1)}|,|Z_{s-}^{
(k+1)}+\Delta_{f_i}^{(k+1)}|\}\leq |Z_{s-}^{
(k+1)}|+|\Delta_{f_i}^{(k+1)}|.
$$
Thus, with ($\mathbf{H_3}$) and Jensen's inequality, Young's inequality, we have the following inequality:
$$
\aligned
\mI_{5,\varepsilon}^{\l}
\leq&\int_0^{t+}e^{-\l s}\left\{3\left(|Z_{s-}^{(k+1)}|^2+\mW_2(\mu_s^{(k)},\mu_s^{(k-1)})^2\right)+2K_3\left(|Z_{s-}^{(k+1)}|^2+\mW_2(\mu_s^{(k)},
\mu_s^{(k-1)})^2\right)\right\}\dif s\\
\leq &(3+2K_3)\int_0^{t}e^{-\l s}\left\{|Z_s^{(k+1)}|^2+\mW_2(\mu_s^{(k)},\mu_s^{(k-1)})^2\right\}\dif s. \\
\endaligned
$$
Combining the above inequalities and
$$
\aligned
\mE\left[\mI_{7,\varepsilon}^{\l}\right]=&\int_0^{t+}e^{-\l s}\int_{U_1}\left[V_{\l,\varepsilon}\left(\left|Z_{s-}^{(k+1)}+\Delta_{f_2}^{(k)}\right|^2\right)
-V_{\l,\varepsilon}(|Z_{s-}^{(k+1)}|^2)\right]\nu_1(\dif u)\dif s\\
\leq &3\int_0^{t}e^{-\l s}\left\{\mE|Z_s^{(k+1)}|^2+\mE\mW_2(\mu_s^{(k)},\mu_s^{(k-1)})^2\right\}\dif s,\\
\endaligned
$$ it holds that
$$
\aligned
&\mE \left[e^{-\l t}V_{\l,\varepsilon}\left(\left|Z_{t}^{(k+1)}\right|^2\right)\right]\\ \leq &[1+(\l+4c_{\l}K_2^2\varepsilon)t]\varepsilon-\left(\l-4K_1-K_2^2-2K_3-6\right)\int_0^{t}e^{-\l s}\mE|Z_s^{(k+1)}|^2\dif s\\
&
+(4K_1+2K_3+6)\int_0^{t}e^{-\l s}\mE|Z_s^{(k)}|^2\dif s,\\
\endaligned
$$
where $\mE\mI_{4,\varepsilon}^{\l}(t)=0$ and $\mE\mI_{6,\varepsilon}^{\l}(t)=0$. In the meantime, the above inequality holds by
$$
 \mE\mW_2(\mu_s^{(k)},\mu_s^{(k-1)})^2\leq \mE|Z_s^{(k)}|^2.
$$

Choosing $\l=0$ and setting $\varepsilon\downarrow 0$, and using Gr\"onwall's inequality, there is some constant $C$ depending on $p,K_1,K_2,K_3$ such that
\beq\label{c-7}
\mE\left|Z_{t}^{(k+1)}\right|^2\leq Ce^{C t}\int_0^{t}\mE|Z_{s}^{(k)}|^2\dif s.
\deq
Use a similar approach of (2.13) in \cite{BH}, for $\l\geq 2K_1e^{1+2K_1T}$ and $t\in[0,T]$, we can obtain
$$
\sup_{0\leq s\leq t}\left[e^{-\l s}\mE(|Z^{(k+1)}_s|^2)\right]\leq e^{-1}\sup_{0\leq s\leq t}\left[e^{-\l s}\mE(|Z^{(k)}_s|^2)\right]\leq e^{-k}\sup_{0\leq s\leq t}\left[e^{-\l s}\mE(|Z^{(1)}_s|^2)\right].
$$
Subsequently, taking the assumption $\mathbf{(H_2)}$, using BDG inequality and Jensen's inequality for $\mE \left[\sup_{0\leq t\leq T}e^{-\l t}V_{\varepsilon}\left(|Z_{t}^{(k+1)}|^2\right)\right]$, and taking $\varepsilon \downarrow 0$, then there is a constant $C_{T}>0$ such that
$$
\mE \left[\sup_{0\leq t\leq T}e^{-\l t}|Z_{t}^{(k+1)}|^2\right]\leq C_{T}\exp\left(-\left(\frac{1}{2}\mathbf{1}_{\{\a=\frac{1}{2}\}}+(2\a-1)\mathbf{1}_{\a\in(\frac{1}{2}, 1]}\right)k\right).
$$
%It follows from Lemma \ref{lem-1}, Fatou's lemma and the monotone convergence theorem that
%$$
%\mE\left[\left|Z_{t\wedge \t_{\theta}}^{(k+1)}\right|^2\right]\leq \lim_{R\rightarrow\infty}\mE\left[\left|Z_{\t}^{(k+1)}\right|^2\right]\leq Ce^{C t}\int_0^{\t}\mE|Z_{s\wedge \t_{\theta}}^{(k)}|^2\dif s,
%$$
%then, we have $\mE\left[\left|Z_{t\wedge \t_{\theta}}^{(k+1)}\right|^2\right]\rightarrow 0$ as $k$ goes to infinity.

%Furthermore, it is to verify that $\lim_{k\rightarrow\infty}\cP(\t_{\theta}\leq t)=0$. Due to Chebyshev's inequality, it holds that
 %\beq\label{c-14}
 %\cP(\t_{\theta}\leq t)\leq \frac{\mE\left[\left|Z_{t\wedge \t_{\theta}}^{(k+1)}\right|^2\right]}{\theta^2}\rightarrow 0 (k\rightarrow\infty).
%\deq
Thus, it is sufficient to prove that $\mE\left[\sup_{0\leq t\leq T}\left|Z^{(k+1)}_{t}\right|^2\right]\rightarrow 0$ as $k$ goes to infinity.

Step $2$: Taking $k\rightarrow\infty$, the iteration scheme (\ref{c-6}) will converge to SDE (\ref{a-1}). Thus, we can check the existence of strong solution of SDE (\ref{a-1}).

With the properties of $\mW-$Wasserstein distance, there is an $(\cF_t)_{t\in[0,T]}$-adapted continuous stochastic process $(X_t)_{t\in[0,T]}$ with
$X_0=\xi$ and $\mu_t=\cL_{X_t}$ such that
$$
\lim_{k\rightarrow\infty}\sup_{t\in[0,T]}\mW_2(\mu_t^{(k)},\mu_t)^2\leq \lim_{k\rightarrow\infty}\mE\|X^{(k)}-X\|_{\infty,T}^2=0.
$$

That is, we are going to verify that
\beq\label{c-15}
\lim_{k\rightarrow\infty}\mE\left[\sup_{0\leq t\leq T}\left|X_t^{(k)}-X_t\right|^2\right]=0.
\deq
 Let $\hat{Z}_t^{(k)}=X_t^{(k)}-X_t$, $\hat{\Delta}_{f_i}^{(k)}=f_i(X_{s-}^{(k)},\mu_s^{(k-1)},u)-f_i(X_{s-},\mu_s,u)$, $i=0,1$. Using It\^o formula to $
 V_{\varepsilon}\left(\left|\hat{Z}_{t}^{(k)}\right|^2\right)$, we find that
 $$
 V_{\varepsilon}\left(\left|\hat{Z}_{t}^{(k)}\right|^2\right):=\sum_{i=1}^6\cD_{i,\varepsilon}^{(k)}(t\wedge \t_R),
 $$
 where
$$
\aligned
\cD_{1,\varepsilon}^{(k)}(t)=&\int_0^{t}DV_{\varepsilon}(|\hat{Z}_{s}^{(k)}|^2)\left(b(X_s^{(k)},\mu_s^{(k-1)})-b(X_s,\mu_s)\right)\dif s\\
\cD_{2,\varepsilon}^{(k)}(t)=&\frac{1}{2}\int_0^{t}D^2V_{\varepsilon}(|\hat{Z}_{s}^{(k)}|^2)\left|\s(X_s^{(k)})-\s(X_{s})\right|^2\dif s\\
\cD_{3,\varepsilon}^{(k)}(t)=&\int_0^{t}DV_{\varepsilon}(|\hat{Z}_{s}^{(k)}|^2)\left(\s(X_s^{(k)})-\s(X_{s})\right)\dif W_s\\
\cD_{4,\varepsilon}^{(k)}(t)=&\int_0^{t+}\int_{U_0}\left[V_{\varepsilon}\left(\left|\hat{Z}_{s-}^{(k)}+\hat{\Delta}_{f_0}^{(k)}\right|^2\right)
-V_{\varepsilon}(|\hat{Z}_{s-}^{(k)}|^2)
-DV_{\varepsilon}(|\hat{Z}_{s-}^{(k)}|^2)\hat{\Delta}_{f_0}^{(k)}\right]
\nu_0(\dif u)\dif s\\
\cD_{5,\varepsilon}^{(k)}(t)=&\int_0^{t+}\int_{U_0}\left[V_{\varepsilon}\left(\left|\hat{Z}_{s-}^{(k)}+\hat{\Delta}_{f_0}^{(k)}\right|^2\right)
-V_{\varepsilon}(|\hat{Z}_{s-}^{(k)}|^2)\right]\widetilde{N}_0(\dif s,\dif u)\\
\cD_{6,\varepsilon}^{(k)}(t)=&\int_0^{t+}\int_{U_2}\left[V_{\varepsilon}\left(\left|\hat{Z}_{s-}^{(k)}+\hat{\Delta}_{f_1}^{(k)}\right|^2\right)
-V_{\varepsilon}(|\hat{Z}_{s-}^{(k)}|^2)\right]N_1(\dif s,\dif u)\\
\endaligned
$$

By the continuity of $b_1(\cdot,\mu)$ for any $\mu\in\cP_1(\mR)$ and the dominated convergence theorem, we infer that
$$
\aligned
\lim_{k\rightarrow\infty}\mE\left[\sup_{0\leq t\leq T}\cD_{1,\varepsilon}^{(k)}\right]
\leq 4K_1\lim_{k\rightarrow\infty}\int_0^{T}\mE\left[\sup_{0\leq s\leq t}\left(|\hat{Z}_s^{(k)}|^2+\mW_2(\mu_s^{(k)},\mu_s)^2\right)\right]\dif t,
\endaligned
$$
where the last inequality holds with (\ref{3-1}).
Taking $\varepsilon\downarrow 0$ and using similar way of $\mI_{3,\varepsilon}^{\l}$, we known
$$
\aligned
&\lim_{k\rightarrow\infty}\mE\left[\sup_{0\leq t\leq T}\cD_{2,\varepsilon}^{(k)}\right]\leq K_2^2\lim_{k\rightarrow\infty}\int_0^{T}\mE\left[\sup_{0\leq s\leq t}|\hat{Z}_s^{(k)}|^2\right]\dif t.
\endaligned
$$
Using BDG's inequality, Jensen's inequality and Young's inequality, and by $\mathbf{(H_2)}$, one sees
$$
\aligned
&\lim_{k\rightarrow\infty}\mE\left[\sup_{0\leq t\leq T}\cD_{3,\varepsilon}^{(k)}\right]\leq4\sqrt{2}K_2\mE\left(\int_0^{T }|X_t^{(k)}-X_t|^{2+2\a }\dif t\right)^{\frac{1}{2}}\\
\leq &\lim_{k\rightarrow\infty}\left\{4\sqrt{2}K_2\lim_{k\rightarrow\infty}\mE\left[\sup_{0\leq t\leq T}|\hat{Z}_t^{(k)}|^2\right]+4\sqrt{2}\left[\int_0^{T}\mE|X_t^{(k)}-X_t|\dif t\right]\mathbf{1}_{\{\a=\frac{1}{2}\}}\right.\\ &\left.+\left[\mE\left(\sup_{0\leq t\leq T}|\hat{Z}_t^{(k)}|\right)+16K_2^2\int_0^{T }\mE|X_t^{(k)}-X_t|^{2\a-1}\dif t\right]\mathbf{1}_{\{\a\in(\frac{1}{2},1]\}}\right\}.\\
\endaligned
$$
%For $|X_t^{(k)}-X_t|> 1$, we can get
%$$
%\aligned
%&\lim_{k\rightarrow\infty}\mE\left[\sup_{0\leq t\leq T}\cD_{3,\varepsilon}^{(k)}\right]\leq4\sqrt{2}K_2|p|\lim_{k\rightarrow\infty}\mE\left(\int_0^{T\wedge \t_N\wedge\t_{\theta}}|X_t^{(k)}-X_t|^{2\a p}\dif t\right)^{\frac{1}{2}}\\
%\leq &4\sqrt{2}K_2|p|\lim_{k\rightarrow\infty}\left(\int_0^{T\wedge \t_N\wedge\t_{\theta}}\mE|X_t^{(k)}-X_t|^{p}\dif t\right)^{\frac{1}{2}}\mathbf{1}_{\{\a=\frac{1}{2}\}}\\
%&+4\sqrt{2}K_2|p|\lim_{k\rightarrow\infty}\left[\mE\|X^{(k)}-X\|_{\infty,T}+16K_2^2\int_0^{T\wedge \t_N\wedge\t_{\theta}}(\mE|X_t^{(k)}-X_t|^{2\a p-1})\dif t\right]\mathbf{1}_{\{\a\in(\frac{1}{2},1]\}}=0. \\
%\endaligned
%$$
 With the assumption $\mathbf{(H_3)}$ and the similarly way of $\mI_{4,\varepsilon}$, it holds that
 $$
\lim_{k\rightarrow\infty}\mE\left[\sup_{0\leq t\leq T}\cD_{4,\varepsilon}^{(k)}\right] \leq (3+2K_3)\lim_{k\rightarrow\infty}\int_0^{T}\mE\left[\sup_{0\leq s\leq t}\left(|\hat{Z}_s^{(k)}|^2+\mW_2(\mu_s^{(k)},\mu_s)^2\right)\right]\dif t=0,
 $$
By BDG's inequality, Jensen's inequality and Young's inequality, we claim that
$$
\aligned
&\lim_{k\rightarrow\infty}\mE\left[\sup_{0\leq t\leq T}\cD_{5,\varepsilon}^{(k)}\right]\\
\leq & \lim_{k\rightarrow\infty}\mE\left[\int_0^{T+ }\int_{U_0}\left[V_{\varepsilon}\left(\left|\hat{Z}_{s-}^{(k)}+\hat{\Delta}^{(k)}_{f_0}\right|^2\right)
-V_{\varepsilon}(|\hat{Z}_{s-}^{(k)}|^2)\right]^2\nu_0(\dif u)\dif s\right]^{\frac{1}{2}}\\
\leq &3\lim_{k\rightarrow\infty}\int_0^{T}\mE\left[\sup_{0\leq s\leq t}\left(|\hat{Z}_s^{(k)}|^2+\mW_2(\mu_s^{(k)},\mu_s)^2\right)\right]\dif t.\\
\endaligned
$$
Because of the properties of the jump, we can get
$$
\aligned
&\lim_{k\rightarrow\infty}\mE\left[\sup_{0\leq t\leq T}\cD_{6,\varepsilon}^{(k)}\right]\\
\leq &\lim_{k\rightarrow\infty} \mE\left[\int_0^{T+}\int_{U_1}V_{\varepsilon}\left(\left|\hat{Z}_{s-}^{(k)}+\hat{\Delta}^{(k)}_{f_1}\right|^2\right)
-V_{\varepsilon}(|\hat{Z}_{s-}^{(k)}|^2)\nu_1(\dif u)\dif s\right]\\
\leq &3\lim_{k\rightarrow\infty}\int_0^{T}\mE\left[\sup_{0\leq s\leq t}\left(|\hat{Z}_s^{(k)}|^2+\mW_2(\mu_s^{(k)},\mu_s)^2\right)\right]\dif t= 0.\\
\endaligned
$$
Thus, (\ref{c-15}) is established.

Combining with above there inequality, taking $k\rightarrow\infty$ in SDE (\ref{c-6}),
by extracting a suitable subsequence, we get $\mP-a.s.$
$$
\aligned
\dif X_t=&b(X_t,\mu_t)\dif t+\s(X_t)\dif W_t+\int_{U_0}f_0(X_{t-},\mu_t,u)\widetilde{N}_0(\dif t,\dif u),\\
&+\int_{U_2}f_1(X_{t-},\mu_t,u)N_1(\dif t,\dif u) \ \ t\in[0,T].\\
\endaligned
$$
That is, $L^2(\Omega,C([0,T];\mR^d))$ is complete, and $X_t\in L^2(\Omega,\mR^d)$ is a solution to (\ref{a-1}) for $t\in[0,T]$. Thus, the existence of strong solution is verified.

Step $3$: we intend to show the uniqueness of (\ref{a-1}). Let $X_t$ and $Y_t$ solve SDE (\ref{a-1}) with the same initial value $\xi$. For $Z_t:=X_t-Y_t$, by
a similar approach of (\ref{c-7}), and for some constant $C$, one has
$$
\mE|Z_{t}|^2\leq Ce^{C t}\int_0^{t}\mE|Z_{s}|^2\dif s,
$$
employing Gr\"onwall's inequality, we can verify the uniqueness.

Step $4$: we prove (\ref{a-2}) in Theorem \ref{thm-1}. Defined the stopping time $\t_n=\inf\{t\geq 0: |X_t|\geq n\}$,
using a similar approach to the Lemma \ref{lem-1}, for any $p\geq 2$, it holds that
$$
\mE\left[\sup_{0\leq t\leq T}\left(|X_{t\wedge \t_n}|^p\right)\right]\leq C_{T}(1+\mE|\xi|^p).
$$

While, it holds that
$$
\mP(\t_n\leq T)=\frac{1}{n^p}\mE\left(|X_{\t_n}|^p\mathbf{1}_{\t_n\leq T\}}\right)\leq\frac{1}{n^p}\mE|X_{T\wedge\t_n}|^p\leq \frac{1}{n^p}C_{T}(1+\mE|\xi|^p).
$$
Summing both sides of the above inequality, we have
$$
\sum_{n=1}^{\infty}\mP(\t_n\leq T)\leq \sum_{n=1}^{\infty}\frac{1}{n^p}C_{T}(1+\mE|\xi|^p)<\infty.
$$
Using Borel-Cantelli lemma, for any $p\geq 2$, (\ref{a-2}) is verified.
At last, with the linear growth condition $\mathbf{(H_4)}$, for any $0<p<2$, again with $\mE[\sup_{t\in[0,T]}|X_t|^p]\leq (1+\mE|\xi|^2)^{p/2}e^{pCT}$, there is a positive constant $C$ such that for any $p>0$,
$$
\mE\left(\sup_{t\in[0,T]}|X_t|^p\right)<\infty.
$$
In the detailed proving process, please see Theorem 4.4 in \cite{Mao}.
That is, the conclusion is established.
\end{proof}
\section{Proof of Theorem \ref{thm-2}}
\begin{lem}\label{lem-2}Assume $\mathbf{(H_1)}$--$\mathbf{(H_4)}$. There exists a constant $C$ independent of $N$, such that for all $T>0$, $N\in\mN$ and $i=1,2,3,\cdots,N$, for any $p>0$, $\mE|\xi_i|^p<\infty$, it holds that
\beq\label{d-1}
\sup_{i=1,2,\cdots,N}\mE\big\|X_t^{N,i}\big\|_{\infty,T}^p\leq C(1+\mE|\xi_i|^p).
\deq
\end{lem}
\begin{proof}For $X_t:=\left(X_t^1,X_t^2,\cdots,X_t^N\right):=(x^1,x^2,\cdots,x^N)\in\mR^N$, let
$$
\widetilde{\mu}_X^N:=\frac{1}{N}\sum_{i=1}^N\delta_{x^i},\ \  \hat{B}(X_t):=\left(b(X_t^1,\widetilde{\mu}_X^N),b(X_t^2,\widetilde{\mu}_X^N),\cdots,b(X_t^N,\widetilde{\mu}_X^N)\right)^*,
$$
$$
\hat{A}(X_t):=\mathrm{diag}\left(\s(X_t^1),\s(X_t^2),\cdots,\s(X_t^N)\right),\ \ \hat{W}_t:=\left(W_t^1,W_t^2,\cdots,W_t^N\right)^*,
$$
$$
\hat{F}_i(X_{t-},u):=\left(f_i(X_{t-}^1,\widetilde{\mu}_X^N,u),f_i(X_{t-}^2,\widetilde{\mu}_X^N,u),\cdots,f_i(X_{t-}^N,\widetilde{\mu}_X^N,u)\right),\ \ i=0,1.
$$
Thus, (\ref{a-9}) can be written as
\beq\label{d-3}
\dif X_t=\hat{B}(X_t)\dif t+\hat{A}(X_t)\dif \hat{W}_t+\int_{U_0}\hat{F}_0(X_{t-},u)\widetilde{N}_1(\dif t,\dif u)+\int_{U_2}\hat{F}_1(X_{t-},u)N_2(\dif t,\dif u).
\deq
Firstly, using the nature of Wasserstein distance,
$$
\frac{1}{N}\sum_{i=1}^N\left(\delta_{X_t^i},\delta_{Y_t^i}\right)=\cC\left(\widetilde{\mu}_X^N,\widetilde{\mu}_Y^N\right),\ \ X_t^i,Y_t^i\in\mR
$$
such that
$$
\mW_2\left(\widetilde{\mu}_X^N,\widetilde{\mu}_Y^N\right)^2\leq \frac{1}{N}\sum_{i=1}^N|X_t^i-Y_t^i|^2.
$$
Secondly, employing Assumption $\mathbf{(H_1)}$--$\mathbf{(H_3)}$, we obtain that for any $\mathbf{x,y}\in\mR^N$, $i=0,1$ and for some constant $\hat{C}_N$,
$$
|\hat{B}(\mathbf{x})-\hat{B}(\mathbf{y})|\leq \hat{C}_N|\mathbf{x}-\mathbf{y}|
$$
$$
|\hat{A}(\mathbf{x})-\hat{A}(\mathbf{y})|^2\leq \hat{C}_N|\mathbf{x}-\mathbf{y}|^{2\a}\ \ \a\in[1/2,1]
$$
$$
\aligned
\int_{U_i}\max\left\{|\hat{F}_i(\mathbf{x},u)-\hat{F}_i(\mathbf{y},u)|^{2},|\mathbf{x}-\mathbf{y}||\hat{F}_i(\mathbf{x},u)-\hat{F}_i(\mathbf{y},u)|\right\}\nu_i(\dif u)\leq & \hat{C}_N|\mathbf{x}-\mathbf{y}|^2,\\
\endaligned
$$such that $\hat{A}$ and $\hat{B}$ are continuous, and $\lim_{h\rightarrow 0}\int_{U_i}|\hat{F}(\mathbf{x}+h,u)-\hat{F}(\mathbf{x},u)|^{2}\nu_i(\dif u)=0$.
On the other hand, Assumption $\mathbf{(H_4)}$ implies for any $\mathbf{x}\in\mR^N$ and $t\in[0,T],$
$$
|\hat{B}(\mathbf{x})|^2\vee|\hat{A}(\mathbf{x})|^2\leq \hat{C}_N(1+|\mathbf{x}|^2)
$$
$$
\int_{U_i}\max\left\{|\mathbf{x}||\hat{F}_i(\mathbf{x},u)|,|\hat{F}_i(\mathbf{x},u)|^2\right\}\nu_i(\dif u)\leq\hat{C}_N|(1+|\mathbf{x}|^2).
$$
The above five inequalities are proved in detail Lemma 4.1 of \cite{LMW2}.
Thereby, by Theorem 175 in \cite{SR} yield that (\ref{d-3}) has a weak solution. At last, using the similar argument to (\ref{c-1}), we can get the uniqueness of strong solution. And by BDG's inequality, H\"older's inequality, Gr\"onwall's inequality and stopping time, (\ref{d-1}) can be verified.
\end{proof}

\begin{cor}\label{cor-1}Under conditions $\mathbf{(H_1)}$--$\mathbf{(H_4)}$ and $\mE|\xi_i|^p<\infty(p>0)$, for all $T>0$, $N\in\mN$ and $i=1,2,3,\cdots,N$, there exists a constant $C$  independent of $N$ such that for any $p>0$,
$$
\mE\left(\sup_{0\leq t\leq T}|X_t^{i}|^p\right)\leq C(1+\mE|\xi_i|^p).
$$
and
$$
\sup_{i=1,2,\cdots,N}\mE\big\|X_t^{h,N,i}\big\|^p_{\infty,T}\leq C(1+\mE|\xi_i|^p).
$$
\end{cor}
\begin{proof}[Proof of Theorem \ref{thm-2}] Let $Z_t^{N,i}=X_t^i-X_t^{N,i}$ and $\Delta_{f_i}^{N,i}=f_i(X_{s-}^{i},\mu^{i}_s,u)-f_i(X_{s-}^{N,i},\mu_s^{N},u)$, $i=0,1$.
By It\^o formula for $\mE[V_{\varepsilon}(|Z_t^{N,i}|^p)](p=1,2)$,
and setting $\eps\downarrow 0$, there exists some constants $C_1, C_2,\hat{C}_1,\hat{C}_2$ dependent of $K_1,K_2,p$, we derive that
\beq\label{4-8}
\aligned
\mE|Z_t^{N,i}|\leq& C_1\int_0^t\mE|Z_s^{N,i}|\dif s+C_2\int_0^t\mE\mW_2(\mu^{i}_s,\hat{\mu}_s^{N})\dif s,\\
\mE|Z_t^{N,i}|^2\leq &\hat{C}_1\int_0^t\mE|Z_s^{N,i}|^2\dif s+\hat{C}_2\int_0^t\mE\mW_2(\mu^{i}_s,\hat{\mu}_s^{N})^2\dif s,\\
\endaligned
\deq
Besides, by triangle inequality and for any $\mu,\nu\in\cP(\mR)$, we can obtain that for any $t\geq 0$,
\beq\label{4-5}
\aligned\mW_2(\mu^{i}_t,\hat{\mu}_t^{N})\leq &\mW_2(\mu^{i}_t,\widetilde{\mu}_t^{N})+\mW_2(\widetilde{\mu}_t^{N},\hat{\mu}_t^{N}) \leq\mW_2(\mu^{i}_s,\widetilde{\mu}_t^{N})+\frac{1}{N}\sum_{j=1}^N|X_t^j-X_t^{N,j}|,\\
\mW_2(\mu^{i}_t,\hat{\mu}_t^{N})^2\leq& \mW_2(\mu^{i}_t,\widetilde{\mu}_t^{N})^2+\mW_2(\widetilde{\mu}_t^{N},\hat{\mu}_t^{N})^2 \leq\mW_2(\mu^{i}_t,\widetilde{\mu}_t^{N})^2+\frac{1}{N}\sum_{j=1}^N|X_t^j-X_t^{N,j}|^2.
\endaligned
\deq
Using lemma \ref{b-1} and Corollary 4.1  with $d=1$, for some $q>2p$ and constants $C$ depending only on $p,q$, implies that
\beq\label{4-6}
\aligned
&\mE\left[\mW_2(\mu^{i}_t,\widetilde{\mu}_t^{N})^2\right]\leq  CN^{-1/2},\\
&\mE\left[\mW_2(\mu^{i}_t,\widetilde{\mu}_t^{N})\right]\leq \left[\mE\left(\mW_2(\mu^{i}_t,\widetilde{\mu}_t^{N})^2\right)\right]^{1/2}\leq CN^{-1/4}.\\
\endaligned
\deq

Thus,  combining with (\ref{4-5}) and (\ref{4-6}), for any $t\in[0,T]$, we have
\beq\label{4-7}
\aligned
\mE\mW_2(\mu^{i}_t,\hat{\mu}_t^{N})\leq&\mE|Z_t^{N,i}|+CN^{-1/4} \\
\mE\mW_2(\mu^{i}_t,\hat{\mu}_t^{N})^2\leq&\mE|Z_t^{N,i}|^2+CN^{-1/2} \\
\endaligned
\deq

Subsequently, using Gr\"onwall's inequality and substituting (\ref{4-7}) into (\ref{4-8}) and the fact that $(Z^{N,j})_{1\leq j\leq N}$ are identically distributed, for any $t\geq 0$ and some constant $C_3,\hat{C}_3>0$, it holds that
$$
\aligned
\mE|Z_t^{N,i}|&\leq C_3\int_0^t\{\mE|Z_s^{N,i}|+N^{-1/4}\}\dif s,\\
\mE|Z_t^{N,i}|^2&\leq \hat{C}_3\int_0^t\{\mE|Z_s^{N,i}|+N^{-1/2}\}\dif s,\\
\endaligned
$$
Employing Burkhold-Davis-Gundy's inequality, Young's inequality and Jensen's inequality, there exist the constants $C_4,C_5,\hat{C}_4$, we can get
$$
\aligned
&\mE\left(\sup_{0\leq t\leq T}|Z_t^{N,i}|\right)\\
\leq & C_3\int_0^T\{\mE|Z_t^{N,i}|+N^{-1/4}\}\dif t+C_3\left(\int_0^T\mE|Z_t^{N,i}|^{2\a }\dif t\right)^{1/2}+C_3\mE\left[\int_0^{T+}\int_{U_0}|\Delta_{f_1}^{N,i}|^2\nu_0(\dif u)\dif t\right]^{1/2}\\
\leq &C_4\int_0^T\{\mE|Z_t^{N,i}|+N^{-1/4}\}\dif t+C_4\left(\int_0^T\mE|Z_t^{N,i}|\dif t\right)^{1/2}\mathbf{1}_{\{\a=1/2\}} \\
& +\left\{\frac{1}{2}\mE\left(\sup_{0\leq t\leq T}|Z_t^{N,i}|\right)+C_5\int_0^T(\mE|Z_t^{N,i}|)^{(2\a -1)}\dif t\right\}\mathbf{1}_{\{\a\in(1/2,1]\}}.\\
\endaligned
$$
and
$$
\aligned
&\mE\left(\sup_{0\leq t\leq T}|Z_t^{N,i}|^2\right)\\
\leq & \hat{C}_3\int_0^T\{\mE|Z_t^{N,i}|+N^{-1/2}\}\dif t+\hat{C}_3\int_0^T\mE|Z_t^{N,i}|^{2\a }\dif t+\hat{C}_3\mE\int_0^{T+}\int_{U_0}|\Delta_{f_1}^{N,i}|^2\nu_0(\dif u)\dif t\\
\leq &\hat{C}_4\int_0^T\{\mE|Z_t^{N,i}|+N^{-1/2}\}\dif t+\hat{C}_4\int_0^T\mE|Z_t^{N,i}|\dif t\mathbf{1}_{\{\a=1/2\}} \\
& +\hat{C}_4\int_0^T\left\{2(1-\a)\mE|Z_t^{N,i}|+(2\a-1)\mE|Z_t^{N,i}|^{2}\right\}\dif t\mathbf{1}_{\{\a\in(1/2,1]\}}.\\
\endaligned
$$

Finally, we infer that
$$
\mE\left(\sup_{0\leq t\leq T}|Z_t^{N,i}|^p\right)\leq C_{T}\left\{
\begin{aligned}
&N^{-1/4}\mathbf{1}_{\{\a=1/2\}}+N^{-\frac{2\alpha-1}{4}}\mathbf{1}_{\{\a\in(1/2,1]\}},\ \ p=1;\\
&N^{-1/2}\mathbf{1}_{\{\a=1/2\}}\ \ p=2.  \\
\end{aligned}
\right.
$$
 Thus, if $N$ is sufficiently large, the conclusion (\ref{1-4}) is established.
\end{proof}
\section{Proof of Theorem \ref{thm-3}}
\begin{lem}\label{lem-3} Suppose that Assumptions $\mathbf{(H_1)}$--$\mathbf{(H_4)}$ hold. Then for any $T>0$, and $h\in(0,1)$ sufficiently small, there is a positive constant $C_T$, we have
$$
\sup_{i=1,2,\cdots,N}\mE\left[\sup_{0\leq t\leq T}\left|X_t^{N,i}-X_t^{h,N,i}\right|\right]<C_T
\left(\ln\frac{1}{h}\right)^{-1/2}\mathbf{1}_{\{\a=\frac{1}{2}\}}+
\left(h^{\frac{(2\a-1)^2}{2}}+h^{\frac{\b(2\a-1)}{2}}\right)\mathbf{1}_{\{\a\in(\frac{1}{2},1]\}}
$$
$$
\sup_{i=1,2,\cdots,N}\mE\left[\sup_{0\leq t\leq T}\left|X_t^{N,i}-X_t^{h,N,i}\right|^2\right]<C_T
\left(\ln\frac{1}{h}\right)^{-1}\mathbf{1}_{\{\a=\frac{1}{2}\}}+
\left(h^{\frac{2\a-1}{2}}+h^{\frac{\b}{2}}\right)\mathbf{1}_{\{\a\in(\frac{1}{2},1]\}}.\\
$$
\end{lem}
\begin{proof}For any $i=1,2,\cdots,N$, set
$$
Z_t^{h,N,i}:=X_t^{N,i}-X_t^{h,N,i},\ \  Z_{t_h}^{h,N,i}:=X_t^{h,N,i}-X_{t_h}^{h,N,i}
$$
and
$$
 \Delta_{f_{\kappa}}^{h,N,i}:=f_{\kappa}(X_{s-}^{N,i},\mu^{N}_s,u)-f_{\kappa}(X_{s_h}^{h,N,i},\mu_{s_h}^{h,N},u)\ \ \kappa=0,1.
$$
Applying It\^o's formula for $V_{\lambda,\varepsilon}(|Z_t^{h,N,i}|)$,
and using a similar proof method of Lemma 3.4 in \cite{BH} and Theorem 1.2, and combining with $\mathbf{(H_1)}$--$\mathbf{(H_3)}$, we infer
$$
\aligned
|Z_t^{h,N,i}|\leq& \eps+\widetilde C_{1}\int_0^t\Big\{|Z_s^{h,N,i}|+|Z_{s_h}^{h,N,i}|+|Z_{s_h}^{h,N,i}|^{\b }+\mW_2(\mu^{N}_s,\mu_{s_h}^{h,N})\Big.\\
&\Big.+\frac{1}{|Z_s^{h,N,i}|\ln \lambda}\mathbf{1}_{\{[\eps/\lambda,\eps]\}}(|Z_s^{h,N,i}|)\left(|Z_s^{h,N,i}|^{2\a }+|Z_{s_h}^{h,N,i}|^{2\a }\right)\Big\}\dif s\\
\endaligned
$$
where $\widetilde C_{1}$ is a positive constant.
Meanwhile, using the following nature of $\mW_p$-Wasserstein distance:
$$
\aligned
\mE\mW_2(\mu^{N}_s,\mu_{s_h}^{h,N})\leq &\mE|Z_s^{h,N,i}|+\mE|Z_{s_h}^{h,N,i}|,\\
\mE\mW_2(\mu^{N}_s,\mu_{s_h}^{h,N})^2\leq &\mE|Z_s^{h,N,i}|^2+\mE|Z_{s_h}^{h,N,i}|^2,\\
\endaligned
$$
By the similar approach of Lemma 3.4 in \cite{BH}, then the conclusion can be obtained directly.
\end{proof}

\begin{proof}[Proof of Theorem \ref{thm-3}]
Combining Theorem \ref{thm-2} and Lemma \ref{lem-3}, the proof of Theorem \ref{thm-3} can be complete.
\end{proof}
\section{Data Availability Statement}
%%%%%%%%%%%%%%%%%%%%%%%%%%%%%%%%%%%%%%%%%
All data, models, and code generated or used during the study appear in the submitted article.

%\acknowledgements{\rm We thank the referees for their time and comments. }


\begin{thebibliography}{99}
\bibitem{ADF}Andreis, L.; Dai Pra, P.; Fischer, M.: McKean-Vlasov limit for interacting systems
with simultaneous jumps. {\it Stochastic Analysis and Applications}, 36 (2018), 960-995.
%\bibitem{B}Bihari, I.: A generalization of a lemma of Bellman and its application to uniqueness problem of differential equations, {\it Acta. Math. Acad. Sci. Hungar.}, 7 (1956), 71-94.
\bibitem{BH}Bao, J.; Huang, X.: Approximations of McKean-Vlasov stochastic differential equations with irregular coefficients. {\it J. Theoret. Probab.}, 35 (2022), no. 2, 1187-1215.
%\bibitem{BHY}Bao, J.; Huang, X.; Yuan, C.: Convergence rate of Euler-Maruyama scheme for SDEs with H\"older-Dini continuous drifts. {\it J. Theoret. Probab.}, 32 (2019), no. 2, 848-871.
\bibitem{BMP}Bauer, M.; Meyer-Brandis, T.; Proske, F.: Strong solutions of mean-field stochastic differential equations with irregular drift. {\it Electron. J. Probab.}, 23 (2018), 1-35.
\bibitem{BLPR}Buckdahn, R.; Li, J.; Peng, S.; Rainer, C.: Mean-field stochastic differential equations and associated
PDEs. {\it Ann. Probab.}, 45 (2017), 824-878.
\bibitem{d}Chaudru de Raynal, P.E.: Strong well-posedness of McKean-Vlasov stochastic differential equation with H\"older drift. {\it Stoch. Process Appl.}, 130 (2020), 79-107.
\bibitem{BW}Bai, Y.; Wanf Y.:European option pricing under the Hybrid Heston-CIR model with jumps. {\it
Journal of LanZhou University of Arts And Science (Natural Sciences Edition)}, 34 (3), (2020) 18-23.
\bibitem{CD}Carmona, R.; Delarue, F.:
Probabilistic analysis of mean-field games.
{\it SIAM J. Control Optim.}, 51 (2013), no. 4, 2705-2734.
%\bibitem{CD1}Carmona, R.; Delarue, F.: Probabilistic theory of mean field games with applications. II. Mean field games with common noise and master equations. {\it Probability Theory and Stochastic Modelling}, 84. Springer, Cham, 2018. xxiv+697 pp.
%\bibitem{CGPS}Carrillo, J.A.; Gvalani, R.S.; Pavliotis, G.A.; Schlichting, A.: Long-time behaviour and phase transitions for the Mckean-Vlasov equation on the torus. {\it Arch. Ration. Mech. Anal.}, 235 (2020), 635-690.
\bibitem{CI}Cox, J.C., Ingersoll, J.E., Ross, S.A.: An intertemporal general equilibrium model of asset prices.
{\it Econometrica}, 53 (1985), 363-384.
\bibitem{CM}Crisan, D.; McMurray, E.: Smoothing properties of McKean-Vlasov SDEs. {\it Probab. Theor. Relat. Fields}, 171 (2018), 97-148.
%\bibitem{CST} Chassagneux J.-F.; Szpruch L.; Tse A.: Weak quantitative propagation of chaos via differential calculus on the space of measures, arXiv preprint arXiv:1901.02556.
%\bibitem{Da}Daniel L.: Hierarchies, entropy, and quantitative propagation of chaos for mean field diffusions, arXiv:2105.02983v1.
\bibitem{D}Dobrushin, R. L.: Prescribing a system of random variables by conditional distributions,{\it Th. Probab. and its Applic.}, 3 (1970), 469.
\bibitem{DQ}Ding, X.; Qiao, H.: Euler-Maruyama approximations for stochastic McKean-Vlasov equations with non-Lipschitz coefficients. {\it J. Theoret. Probab.}, 34 (2021), no. 3, 1408-1425.

%\bibitem{dST}Dos Reis, G.; Smith, G.; Tankov, P.: Importance sampling for McKean-Vlasov SDEs, arXiv:1803.09320.
\bibitem{dSW}dos Reis, G., Salkeld, W., Tugaut, J.: Freidlin-Wentzell LDP in path space for McKean-Vlasov equations
and the functional iterated logarithm law. {\it Ann. Appl. Probab.}, 29 (2019), 1487-1540.
%\bibitem{EGZ}Eberle, A.; Guillin, A.; Zimmer, R.: Quantitative Harris-type theorems for diffusions and McKeanVlasov processes. {\it Trans. Am. Math. Soc.}, 371 (2019), 7135-7173.

\bibitem{FG}Fournier, N.; Guillin, A.: On the rate of convergence in Wasserstein distance of the empirical measure. {\it Probab. Theory Related Fields}, 162 (2015), no. 3-4, 707-738.
\bibitem{FL} Fu, Zongfei; Li, Zenghu Stochastic equations of non-negative processes with jumps. {\it Stochastic Process. Appl.}, 120 (2010), no. 3, 306-330.
\bibitem{G}G\"artner, J.: On the McKean-Vlasov limit for interacting diffusions. {\it Math. Nachr.},  137 (1988), 197-248.
\bibitem{Gr}Graham, C.:
McKean-Vlasov It\^o-Skorohod equations, and nonlinear diffusions with discrete jump sets.
{\it Stochastic Process. Appl.}, 40 (1992), no. 1, 69-82.
%\bibitem{GA} Govindan, T.E.; Ahmed, N.U.: On Yosida approximations of McKean-Vlasov type stochastic evolution equations. {\it Stoch. Anal. Appl.}, 33 (2015), no. 3, 383-398.

%\bibitem{HMC}Huang M.; Malham$\mathrm{\acute{e}}$ R.P.; Caines P.E.: Large population stochastic dynamic games: closed-loop McKean-Vlasov systems and the Nash certainty equivalence principle, {\it Communications in Information and Systems}, 6 (2006), no. 3, 221-252.
\bibitem{HRZ}Hao, Z.; R\"ockner, M., Zhang, X.: Strong convergence of propagation of chaos for McKean-Vlasov SDEs with singular interactions. {\it Potential Anal.}, doi.org/10.1007/s11118-023-10101-9.
%\bibitem{HW}Huang, X.; Wang, F.-Y.: Distribution dependent SDEs with singular coefficients. {\it Stoch. Process. Appl.}, 129(2019), 4747-4770.
\bibitem{HW}Huang, X.; Wang, F.-Y.: McKean-Vlasov SDEs with drifts discontinuous under Wasserstein distance.
{\it Discrete Contin. Dyn. Syst.}, 41 (2021), no. 4, 1667-1679.
%\bibitem{JW}Jabin, P.-E.; Wang, Z.: Mean field limit for stochastic particle systems. Active particles. Vol. 1. Advances in theory, models, and applications, 379-402, {\it Model. Simul. Sci. Eng. Technol.}, Birkh\"auser/Springer, Cham, 2017.
%\bibitem{JW1}Jabin, P.-E.; Wang, Z.: Quantitative estimates of propagation of chaos for stochastic systems with $W^{1,\infty}$ kernels. Invent. Math. 214 (2018), no. 1, 523-591.
\bibitem{K}Kac M.: Foundations of kinetic theory, {\it Proceedings of the third Berkeley symposium on mathematical
statistics and probability}, 1954-1955, vol. III, pp. 171-197.
%\bibitem{K1}Kac, M.: Foundations of kinetic theory. Proceedings of the Third Berkeley Symposium on Mathematical Statistics and Probability, 1954-1955, vol. III, pp. 171-197. University of California Press, Berkeley and Los Angeles, Calif., 1956.
%\bibitem{LM}Li, J.; Min, H.: Weak solutions of mean-field stochastic differential equations and application to zero-sum stochastic differential games. {\it SIAM J. Control Optim.}, 54 (2016), 1826-1858.
\bibitem{LMW}Liang M.; Majka, M. B.;Wang, J.: Exponential ergodicity for SDEs and McKean-Vlasov processes with L\'evy noise. {\it Ann. Inst. Henri Poincar$\acute{e}$ Probab. Stat.}, 57 (2021), no. 3, 1665-1701.
\bibitem{LMW2}Li, Y.; Mao, X.; Song, Q.; Wu, F.; Yin, G.: Strong convergence of Euler-Maruyama schemes for McKean-Vlasov stochastic differential equations under local Lipschitz conditions of state variables. {\it IMA J. Numer. Anal.}, 43 (2023), no. 2, 1001-1035.
\bibitem{Mao}Mao, X.: Stochastic differential equations and applications. Second edition.{\it Horwood Publishing Limited, Chichester}, 2008. xviii+422 pp.
\bibitem{M}McKean, H.P.: A class of Markov processes associated with nonlinear parabolic equations. {\it Proc. Nat. Acad. Sci. U.S.A.}, 56 (1966), 1907-1911.
%\bibitem{MBKM}Mezerdi, M.A.; Bahlali, K.; Khelfallah, N.; Mezerdi, B.: Approximation and generic properties of McKean-Vlasov stochastic equations with continuous coefficients, arXiv: 1909.13699.
%\bibitem{MV}Mishura, Y. S.; Veretennikov, A. Y.: Existence and uniqueness theorems for solutions of McKean-Vlasov stochastic equations.
%{\it Theory Probab. Math. Statist.}, (2020), no. 103, 59-101.

\bibitem{RZh}R\"ockner, M.; Zhang, X.: Well-posedness of distribution dependent SDEs with singular drifts. {\it Bernoulli}, 27 (2021), no. 2, 1131-1158.
\bibitem{RRW}Ren, P.; R\"ockner, M.; Wang, F.-Y.: Linearization of nonlinear Fokker-Planck equations and applications, {\it J. Differential Equations}, 322 (2022), 1-37.
\bibitem{RWW}Ren, J.; Wang, L.; W, Z.: Principle of large deviation of McKean-Vlasov SDEs, under review.
 \bibitem{OS}${\O}$ksendal, B.; Sulem, A.: Applied stochastic control of jump diffusions. 3nd Edition, {\it Springer, Cham}, (2019).
 \bibitem{PW}Peter, C.; Wim, S.: Hedging under the Heston model with  jump-to-default. {\it International Journal of Theoretical and Applied FinanceVol.} 11, (2008), no. 04, 403-414.
\bibitem{S}Sznitman, A.S.: Topics in propagation of chaos. In \'Ecole $\mathrm{d}'$\'Et\'e de Pro. de Saint-Flour XIX-1989, {\it Lecture Notes in Math.}, 165-251.
\bibitem{Sh}Shreve, S.E.: Stochastic calculus for finance II: Continuous-Times Models, {\it Spring-Verlag New York}, (2004).
\bibitem{SR}Situ, R.: Theory of stochastic differential equations with jump and applications. Springer, New York (2005).
\bibitem{T}Tankov, P.: Financial modelling with jump processes, {\it Chapman and Hall/CRC}, (2003).
%\bibitem{V}Vlasov, A.A.: The vibrational properties of an electron gas. {\it Sov. Phys.}, Usp. 10 (1968), 721.
%\bibitem{W} Wang, F.-Y.: Distribution dependent SDEs for Landau type equations. {\it Stochastic Process. Appl.}, 128 (2018), 595-621.
\bibitem{W1} Wang, F.-Y.: Exponential ergodicity for non-dissipative McKean-Vlasov SDEs. {\it Bernoulli}, 29 (2023), no. 2, 1035-1062.
\bibitem{WRM}Wang, Z., Ren J., Miao Y.: Euler's scheme of McKean-Vlasov SDEs with non-Lipschitz. arXiv: 2202.08422.
\bibitem{X} Xavier, E.: Well-posedness and propagation of chaos for McKean-Vlasov equations with jumps and locally Lipschitz coefficients.
{\it Stochastic Process. Appl.}, 150 (2022), 192-214.
%\bibitem{XXZZ}Xia, P.; Xie, L.; Zhang, X.; Zhao, G.: Lq(Lp)-theory of stochastic differential equations. {\it Stochastic Process. Appl.}, 130 (2020), no. 8, 5188-5211.
%\bibitem{Zh1}Zhang X.: Homeomorphic flows for multi dimensional SDEs with non-Lipschitz coefficients. {\it Stochastic Process. Appl.}, 115 (2005) 435-448.
%\bibitem{Zh}Zhang X.: Euler-Maruyama approximations for SDEs with non-Lipschitz coefficients and applications. {\it J. Math. Anal. Appl.}, 316 (2006), no. 2, 447-458.
\bibitem{Zh2}Zhang, X.: A discretized version of Krylov's estimate and its applications. {\it Electron. J. Probab.}, 24 (2019), no. 131, 17 pp.

\end{thebibliography}
\end{document}